\documentclass[letterpaper, 10 pt, conference, twocolumn]{ieeeconf}

\IEEEoverridecommandlockouts
\usepackage{graphics}
\usepackage{epsfig}
\usepackage{amsmath}
\usepackage{amssymb}
\usepackage{bbm}
\usepackage{eufrak}
\usepackage[noadjust]{cite}

\newtheorem{lemma}{Lemma}
\newtheorem{theorem}{Theorem}

\newtheorem{definition}{Definition}

\usepackage{comment}
\usepackage{mathtools}
\usepackage{soul}
\usepackage{tikz}
\usepackage{color}
\usepackage{xcolor}
\usepackage{tcolorbox}
\usepackage{float}
\usepackage{subfig}
\usepackage{algorithm}
\usepackage[noend]{algpseudocode}
\usepackage[noabbrev]{cleveref}
\usepackage{scalerel,amssymb}

\def\msquare{\mathord{\scalerel*{\Box}{gX}}}

\makeatletter
\def\BState{\State\hskip-\ALG@thistlm}
\makeatother

\renewcommand{\vec}[1]{\mathbf{#1}}
\graphicspath{{Images/}}

\title{\LARGE \bf Influence Models on Layered Uncertain Networks: \\ A Guaranteed-Cost Design Perspective}

\author{Siavash Alemzadeh and Mehran Mesbahi
	\thanks{The authors are with the William E. Boeing Department of
		Aeronautics and Astronautics, University of Washington, Seattle WA, USA {\tt\small \{alems,mesbahi\}@uw.edu}.
		The research of the authors has been supported by NSF grant SES-1541025 and	ARO Grant W911NF-13-1-0340.}
}

\begin{document}
	
	\maketitle
	\thispagestyle{empty}
	\pagestyle{empty}


	\begin{abstract}
		Control and estimation on large-scale social networks often necessitate the availability of models for the interactions amongst the agents.
		However characterizing accurate models of social interactions pose new challenges due to inherent complexity and unpredictability.
		Moreover, model uncertainty becomes more pronounced for large-scale networks.
		For certain classes of social networks, the layering structure allows a compositional approach.
		In this paper, we present such an approach to determine performance guarantees on layered networks with inherent model uncertainties.
		A factorization method is used to determine robust stability and performance and this is accomplished by a layered cost-guaranteed design via a layered
		Riccati-type solver, mirroring the network structure.
		We provide an example of the proposed methodology in the context of opinion dynamics on large-scale social networks.
		
		\noindent \\ Keywords: \emph{social networks, distributed layered control, guaranteed-cost LQ performance, graph products}
		
	\end{abstract}
	
	\section{INTRODUCTION}
	\label{sec:intro}
	
	Communities are formed by a large number of local and global interactions, linked by a wide range of social and economic interdependencies.
	Systematic understanding of the evolution of communities can be achieved by understanding the influence of its members on each other as well as the role of external factors.
	Such an understanding often requires an accurate model that captures the interactive behaviors \cite{proskurnikov2017tutorial}.
	With the introduction of \emph{sociogram} in 1930's \cite{moreno1934shall}, graphical models of interconnections among a group of individuals was adopted to examine the evolution of communities. 
	%
	This line of work led to several branches in social and behavioral research such as the interdisciplinary science of Social Networks Analysis (SNA).
	The development of these dynamic models has provided an intellectual bridge between the communal social and behavioral interdependencies on one hand, and techniques in system sciences such as control and estimation, on the other \cite{cremonini2017controllability}.
	In the meantime, the advent of large-scale modeling techniques due to complexities of the interdependencies, increasing population size, and the corresponding datasets, have led to the need to revisit algorithms and solution strategies for network-level control and estimation.
		
	In spite of the complexity and unpredictability of large-scale social interactions, characterizing reliable models for these interactions are promising in cases where prior knowledge about the underlying structures of these systems is available.
	In particular, for certain types of social networks, the layering structure allows a compositional approach for the mathematical representation of the system.
	The layering structure in a social network can be induced by a variety of motives such as the presence of distinct social types, geographical coordinates, and financial or political ties.
	The idea of compositional study of a layered system can be compared to distributed systems analysis in the sense that the problem is split into manageable subproblems that can be subsequently solved independently \cite{BAKULE200887,borrelli2008distributed}.
	%
	For example, decomposition of consensus-type networks leads to examining the protocol for each layer~\cite{nguyen2007factorization}.
	Furthermore, \cite{chapman2014controllability} provides a controllability and observability analysis on large-scale composite networked systems based on their factors.
	
	Despite the many advantages of a decompositional approach, the high dimension of the system poses new challenges primary due to the layers' uncertainty as well as perturbations to the layering structure as a whole \cite{ghaedsharaf2017eminence}.
	In the context of social networks these uncertainties may be due to inaccurate modeling of the nature of the interactions,\footnote{For example, due to mis-classification of antagonistic interactions in signed networks \cite{altafini2013consensus}.} as well as whether or not two social entities are directly interdependent.
	Such uncertainties pose difficulties for the control and estimation of such systems.
	As an example, the adoption of a linear quadratic (LQ) theory in social networks, is not only hindered by high dimensionality, but also by inherent model uncertainties. 
	As a result, the strong robustness properties of say, the Linear Quadratic Regulator (LQR) approach, can vanish where small changes to the system parameters lead to instabilities.
	In this paper, we present a compositional method to characterize  performance guarantees on layered social networks with model uncertainties.
	%
	The corresponding distributed analysis and control presented in this work is closely related to \cite{chapman2017data}, where a composite LQR solution is derived from the parameters of the two layers.
	In this work, we obtain sufficient conditions for the robust stability of the composite network based on a layered control mechanism.

	The paper is organized as follows:
	In \S\ref{sec:math-prel}, we provide the mathematical background for our work.
	We introduce the problem setup in \S\ref{sec:prob-setup};  \S\ref{sec:analysis} provides our analysis for synthesizing the controller that leads to performance guarantees on layered uncertain network models.
	We conclude the paper in \S\ref{sec:sims} with an illustrative example over an interdependent social network and concluding remarks are provided in \S\ref{sec:conclusion}.
	
	
	\section{Mathematical Preliminaries}
	\label{sec:math-prel}
	
	A square matrix $N\in\mathbb{R}^{n\times n}$ is \emph{symmetric} if $N^{\top}=N$;
on the other hand, a matrix-valued function $f(N)$ is \emph{symmetric} if $[f(N)]^{\top} = f(N)$.
	The identity matrix is denoted by $\textsc{I}_n$.
	$|T|$ denotes the entry-wise nonnegative version of the matrix $T$.
	For any operator $\odot$, we define $\overset{k}{\underset{i=1}{\odot}} t_i= t_1\odot\dots\odot t_k$.	
	The notation $N\succ0$ ($\succeq0$) is used when $N$ is a positive-(semi) definite matrix, i.e., $x^{\top}Nx>0$ ($\geq0$) for all $x\neq0$; $A\succ B$ if $A-B\succ0$.
	By $P=Q\Lambda Q^{\top}$, we denote the eigendecomposition of the symmetric matrix $P$, where $Q$ and $\Lambda$ represent the eigenvectors and eigenvalues of $P$, respectively.
	The \emph{Cholesky decomposition} of the symmetric matrix $F\succ0$ is given by $F=LL^{\top},$ where $L$ is lower-triangular with real, positive diagonal entries.
	We call $(A,B)$ \emph{controllable}, if and only if the controllability matrix $\mathcal{C}=[B\quad AB\ \dots\ A^{n-1}B]$ has full-rank. On the other hand, $(A,C)$ is \emph{observable} if and only if the pair $(A^{\top},C^{\top})$ is controllable.
	
	A \emph{graph} is characterized by the 3-tuple $\mathcal{G}=(\mathcal{V},\mathcal{E},\mathcal{W})$, where $\mathcal{V}$ is the set of nodes, $\mathcal{E}\subseteq \mathcal{V}\times\mathcal{V}$ denotes the set of edges, and $\mathcal{W}$ consists of weights assigned to edges.
	An edge is said to exit $i$ for $j$ if $(i,j)\in \mathcal{E}$ with edge weight $w_{ij}\in \mathcal{W}$.
	The adjacency matrix of $\mathcal{G}$, denoted by $\mathcal{A(\mathcal{G})}$ is the matrix $[\mathcal{A(\mathcal{G})}]_{ij}=w_{ij}$ when $(i,j)\in \mathcal{E}$ and $[\mathcal{A(\mathcal{G})}]_{ij}=0$ otherwise.
	The graph \emph{Laplacian} is defined as $\mathcal{L} = \mathcal{D} - \mathcal{A}$, where $\mathcal{D}$ is the diagonal \emph{degree} matrix with $\mathcal{D}_{ii} = \sum_{j:(i,j)\in\mathcal{E}} \mathcal{A}_{ij}$.\footnote{The notational dependency of the adjacency, degree, and Laplacian matrices on the graph $\mathcal{G}$ will be suppressed when implicit in the context.}
	The consensus dynamics is defined as $\dot{x} = -\mathcal{L}x$.
	
	The Cartesian product of $m$ graphs
	is denoted by $\mathcal{G}=\overset{m}{\underset{i=1}{\msquare}} \mathcal{G}_i$ where the vertex set of $\mathcal{G}$ has the form $\mathcal{V}_{\mathcal{G}}=\mathcal{V}_1\times\dots\times\mathcal{V}_m$ and the nodes $P_v=(v_1,\dots,v_m)$ and $P_u=(u_1,\dots,u_m)$ are connected if and only if there exists some $i$ such that $(v_i,u_i)\in\mathcal{E}_i$ and $v_j=u_j$ for $j\neq i$ \cite{imrich2000product}.

	The Kronecker product of $A\in\mathbb{R}^{p_1\times q_1}$ and $B\in\mathbb{R}^{p_2\times q_2}$ is an $p_1q_1\times p_2q_2$ matrix denoted by,
	\begin{align*}
		A\otimes B = \begin{bmatrix}
			a_{1,1}B & a_{1,2}B & \dots  & a_{1,q_1}B \\
			a_{2,1}B & a_{2,2}B & \dots  & a_{2,q_1}B \\
			\vdots  & \vdots  & \ddots & \vdots  \\
			a_{p_1,1}B & a_{p_1,2}B & \dots  & a_{p_1,q_1}B
		\end{bmatrix}.
	\end{align*}
	Some important properties of the Kronecker product include the \emph{mixed-product} property, $(A\otimes B)(C\otimes D)=AC\otimes BD$, \emph{distributivity}, $A\otimes (B+C)=(A\otimes B)+(A\otimes  C)$, and \emph{associativity}, $A\otimes (B\otimes C)=(A\otimes B)\otimes C$.
	Moreover, $(\overset{m}{\underset{i=1}{\otimes}} R_i)^{\top}=\overset{m}{\underset{i=1}{\otimes}} R_i^{\top}$ and $(\overset{m}{\underset{i=1}{\otimes}} T_i)^{-1}=\overset{m}{\underset{i=1}{\otimes}} T_i^{-1}$ for invertible $T_i$'s.
	The Kronecker sum is defined on square matrices $M\in\mathbb{R}^{m\times m}$ and $N\in\mathbb{R}^{n\times n}$ as $M\oplus N=M\otimes I_n+I_m\otimes N$.
	The Kronecker product preserves positive (semi-) definiteness.
	
	\section{Problem Setup}
	\label{sec:prob-setup}
	
	%
	Specific classes of large-scale social networks can be modeled, at least approximately, via a layered structure representing interdependent subsystems.
	One may then aim to characterize the properties of the system via those of its factors or layers.
	This decompositional approach is effective for various classes of social and economic networks, where for example, inter-nodal influences among distinct groups lead to opinions on a sequence of issues \cite{jia2015opinion} (Fig. \ref{fig:socialNetworks}).
	The evolutionary study of interconnections among political parties in elections or the investigation of financial ties between different branches of an international organization are two examples of the layered structures in behavioral sciences.
	\begin{figure}[H]
		\centering
		\includegraphics[scale=0.33]{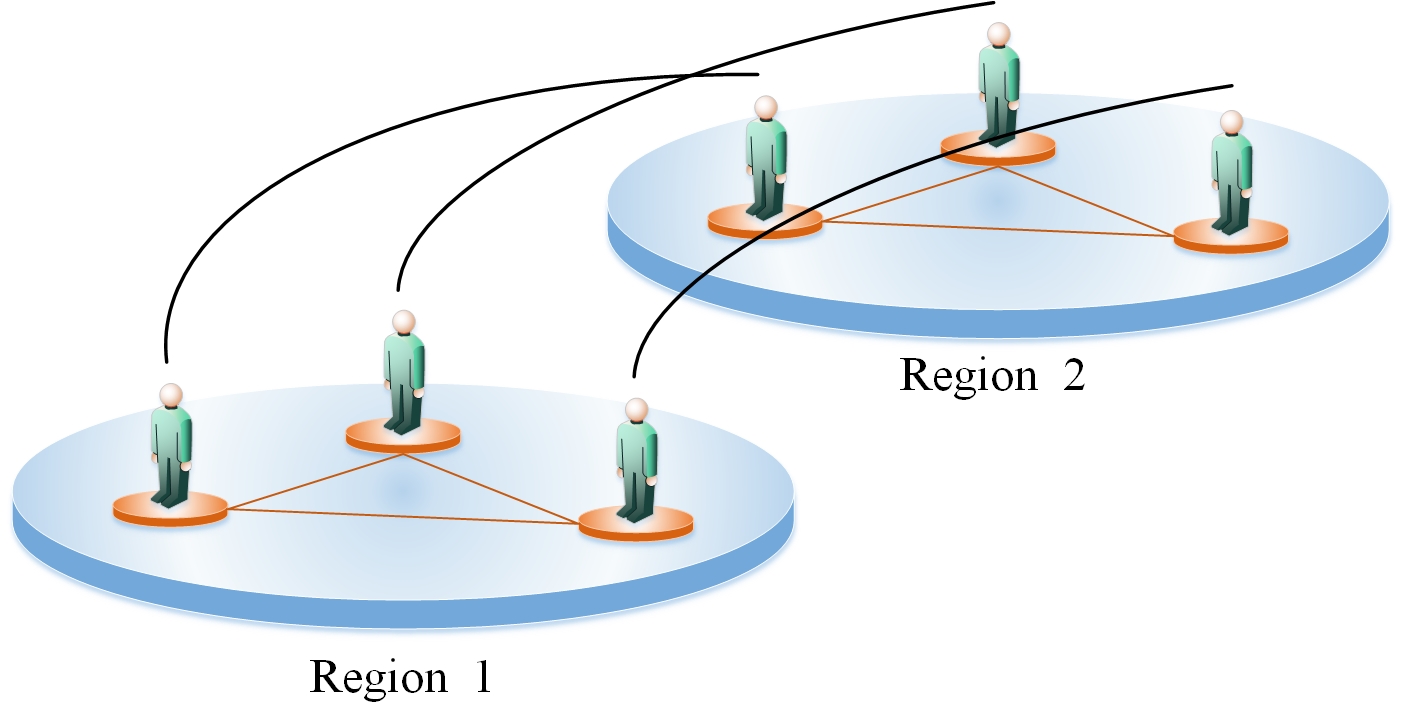}
		\caption{A example of a layered social networks due to geographical distributions}
		\label{fig:socialNetworks}
	\end{figure}
	%
	This types of system representation make it possible to embed more structure into the system and use this embedding to simplify the subsequent computational and theoretic analysis.
	%
	
	\subsection{Dynamics of each layer of the network}
		
	We assume that the evolution of opinions is captured by Taylor's model of influenced attitude change \cite{taylor1968towards}.
	The model considers the change in attitudes of a set of individuals as a result of influence processes within the set, as well as the exposure to external sources.	
	Based on this model, the opinion dynamics of an individual $p$ in a networked system with $n$ agents and $m$ external inputs can be represented as,
	\begin{align*}
		\dot{y}_p(t) = \sum_{q=1}^n a_{pq} (y_q(t)-y_p(t)) + \sum_{k=1}^m b_{pk} (s_k-y_p(t)),
	\end{align*}
	where $y_p$ is the state of $p$, $s_k$ is the $k$th external input, $a_{pq}$ captures the influence between agents $p$ and $q$ and $b_{pk}$ defines the interaction between agent $p$ and the $k$th external (static) input source.
	If the individual $r$ is not directly influenced by an external input, $b_{rk}=0$ for all $k$.
	Particularly, the input to a social organization may be due to a stationary source of communication such as mass media or an influential administrative center.
	In the matrix form, Taylor's dynamical model for layer $i$ of the network assumes the form,
	\begin{align*}
		\dot{x}_i(t) = A_ix_i(t) + B_iu_i(t) \hspace{5mm} \forall i\in\{1,2,\dots,m\},
	\end{align*}
	where $A_i\in\mathbb{R}^{n_i\times n_i}$ is equivalent to graph \emph{Laplacian},\footnote{The notation is due to the applicability of the methods in this paper for any general linear dynamics and shall not be confused with the adjacency matrix $\mathcal{A}$ of the network.} capturing the difference in attitudes and $B_i\in\mathbb{R}^{n_i\times p_i}$ defines the control ``knob" for the external inputs.	
	Following \cite{chapman2014controllability}, for this type of dynamics,
	\begin{align*}
		A(\overset{m}{\underset{i=1}{\msquare}} \mathcal{G}_i) = \overset{m}{\underset{i=1}{\oplus}} A_i,
	\end{align*}
	where for simplicity $A_i = A(\mathcal{G}_i)$.
	Then the overall network is formed as,
	\begin{align}
		\dot{\vec{x}}(t) = \big( \overset{m}{\underset{i=1}{\oplus}} A_i \big) \vec{x}(t) + \big( \overset{m}{\underset{i=1}{\otimes}} B_i \big) \vec{u}(t),
		\label{eq:overallDynamics1}
	\end{align}
	where $\overset{m}{\underset{i=1}{\oplus}} A_i \in \mathbb{R}^{\overset{m}{\underset{i=1}{\prod}} n_j \times \overset{m}{\underset{i=1}{\prod}} n_j}$ and $\overset{m}{\underset{i=1}{\otimes}} B_i \in \mathbb{R}^{\overset{m}{\underset{i=1}{\prod}} n_j \times \overset{m}{\underset{i=1}{\prod}} p_j}$.
	This dynamics can also be formulated in discrete-time as well~\cite{chapman2014controllability}.
	
	\subsection{Guaranteed LQ Performance}
	\label{sec:guaranteedLQ}
		
	Perturbations can be induced in social networks due to distortions in existence, nature, or intensity of interactions among the individuals.
	This model uncertainty can eventually lead to instabilities in a social influence model.
	%
	Unstable behavior in social networks generally have unfavorable ramifications such as the advent of clustering or community cleavage \cite{de2017controllability,askarzadeh2017stability,friedkin2015problem}.

	In this work, we employ the LQ theory as a potential methodology to design state-feedback controllers for systems with layered structures.
	%
	LQ methods have been applied in the literature for the control of large-scale systems and social networks \cite{xue2016optimal,tembine2013robust}.
	In general, the applicability of the LQ framework is reasonable when the resources used for
	social influence are restricted.\footnote{For example, when there is cap on the advertisement budget.}
	In such a setting, an LQ regulator can be used to attenuate the effect of the undesirable
	external influences through minimal adjustments in the control variables.
	%
	However, it is well-known that the stability margins of the LQ design do not guarantee robustness to variations in system parameters \cite{chang1972adaptive}.
	It is thus desirable to enhance the LQ design in order to obtain guarantees on the stability
	and performance of the system.
	One approach to achieve this is through extending the \emph{Algebraic Riccati Equation} (ARE); the baseline
	form of the setup assumes the form,
	\begin{align}
		A^{\top}P + PA^{\top} + Q - PBR^{-1}B^{\top}P = 0,
		\label{eq:ARE}
	\end{align}
	for given $Q\succeq 0, R\succ 0$, leading to the LQR optimal controller as $K=R^{-1}B^{\top}P$.
	It is known that robustness to variations in system parameters can be handled by an additional term in \eqref{eq:ARE}.
	We will utilize this methodology to ensure guarantees on the large-scale system performance in presence of uncertainties.		
	Accordingly, appending the uncertainties to the dynamics in \eqref{eq:overallDynamics1} results in,
	\begin{align*}
		\dot{\vec{x}}(t) = \overset{m}{\underset{i=1}{\oplus}} \big( A_i + \Delta A_i \big) \vec{x}(t) + \big( \overset{m}{\underset{i=1}{\otimes}} B_i \big) \vec{u}(t),
	\end{align*}
	where $\Delta A_i$ denotes the uncertainty of the model in layer $i$.
	There are many different structures suggested for $\Delta A_i$ in the literature \cite{dorato1995linear}.
	One common choice of these structured perturbations is,
	\begin{align}
		\Delta A = \sum_{j=1}^d w_j \tilde{A}_j,
		\label{eq:perturbation}
	\end{align}
	for given $\tilde{A}_j$.
	%
	Nevertheless, the results in the paper are derived for the general form of the uncertainty $\Delta A$ of each layer.
	
	%
	The layered structure is one example where compositional control is feasible by applying similar inputs to the network layers.
	In this case, the generalized input matrix can be written as,
	\begin{align*}
		B_{\otimes} = B_1\otimes I_{n_2}\otimes\dots\otimes I_{n_m}.
	\end{align*}
	This assumption helps reduce the intra-layered couplings.
	Hence the main analysis in this work is building upon the following generalized dynamics,
	\begin{align}
		\dot{\vec{x}}(t) = \overset{m}{\underset{i=1}{\oplus}} \big( A_i + \Delta A_i \big) \vec{x} + B_{\otimes} \vec{u}(t).
		\label{eq:overallDynamics3}
	\end{align}
	Our goal is to find a generalized structured controller to achieve an upper bound on the LQR performance index for system \eqref{eq:overallDynamics3}.

	
	\section{Analysis}
	\label{sec:analysis}
	
	In this section we propose a framework for guaranteed performance design for the $m$-layered dynamical system \eqref{eq:overallDynamics3} using an LQ approach.	
	To make the paper self-contained, we restate a theorem from \cite{dorato1995linear} that our main result is built upon.
	We then generalize the sufficient conditions for the layered case.
	
	
%
%
	\vspace{2mm}
	
	\begin{theorem}
		\label{thm:initialControl}
		
		\cite{dorato1995linear} Consider the perturbed dynamical system,
		\begin{align*}
			\dot{x} = (A + \Delta A)x + Bu,
		\end{align*}
		and define the quadratic performance measure,
		\begin{align}
			J=\int_0^{\infty} \big( x^{\top}Qx + u^{\top}Ru \big) dt.
			\label{eq:cost}
		\end{align}
		Let $x_0$ be the initial state and $P\succ 0$ be the solution to the modified ARE,
		\begin{align}
			A^{\top}P + PA + Q - PBR^{-1}B^{\top}P + \mathcal{U}(P) = 0,
			\label{eq:riccatiRootControl}
		\end{align}
		where $\mathcal{U}(.)$ is a positive symmetric function for which,
		\begin{align}
			\Delta A^{\top}P + P\Delta A \preceq \mathcal{U}(P).
			\label{eq:ineq1}
		\end{align}
		Then the feedback control law defined as $u=-Kx$ leads to $J\leq x_0^{\top}Px_0,$ where $K=R^{-1}B^{\top}P$.
	\end{theorem}
	
	\vspace{2mm}
	
	
%

	In \Cref{thm:initialControl}, choosing $\mathcal{U}$ is dictated by a trade-off between the complexity of the design and analytical properties of the solution strategy.
	The choice, however, depends on the nature of the perturbations.
	For instance, the structure given in \eqref{eq:perturbation} implies that,
	\begin{align*}
		\Delta A^{\top} P + P \Delta A
		&= \sum_{j=1}^d w_j \big( \tilde{A}_j^{\top} P + P \tilde{A}_j \big).
	\end{align*}
	One suggested form of $\mathcal{U}$ induced by this type of perturbation is \cite{chang1972adaptive},
	\begin{align}
		\label{eq:defineU}
		\mathcal{U}(P) = \sum_{j=1}^d Q_j|\Lambda_j|Q_j^{\top},
	\end{align}
	where $Q_j$ and $\Lambda_j$ are obtained from the eigendecomposition of the symmetric matrix $\tilde{A}_j^{\top} P + P \tilde{A}_j$ as,
	\begin{align*}
		\tilde{A}_j^{\top} P + P \tilde{A}_j = Q_j\Lambda_jQ_j^{\top}.
	\end{align*}
	It is straightforward to check that \eqref{eq:ineq1} holds under such a definition of $\mathcal{U}$.
	Nonetheless, the analysis in this paper is not limited to any specific types of $\mathcal{U}$.

	\vspace{2mm}
	
	\begin{definition}
		\label{def:notation}
		Given matrices $D$ and $C_i$ for $i\in\{1,2,\dots,\ell\}$, we define,
		\begin{align}
			C_{D}^{\otimes_{k,\ell}} = C_1\otimes\dots\otimes C_{k-1}\otimes D\otimes C_{k+1}\otimes\dots\otimes C_\ell,
			\label{eq:define}
		\end{align}
		i.e., $C_{D}^{\otimes_{k,\ell}}$ replaces $C_k$ with $D$ in $\overset{\ell}{\underset{i=1}{\otimes}} C_i$.
	\end{definition}
	\vspace{2mm}
	
	\Cref{def:notation} is followed by some useful properties that is presented in the following.
	
	\vspace{2mm}
	
	\noindent \textbf{Proposition 1.} Given Definition~\ref{def:notation}, the following hold:
	
	\vspace{2mm}
	
	\begin{enumerate}
		\item $X_{Y}^{\otimes_{k,\ell}} \pm X_{Z}^{\otimes_{k,\ell}} = X_{Y \pm Z}^{\otimes_{k,\ell}}$, 
		\item $(X_{Y}^{\otimes_{k,\ell}})(V_{W}^{\otimes_{k,\ell}}) = (XV)_{(YW)}^{\otimes_{k,\ell}}$,
		\item $(X_{Y}^{\otimes_{k,\ell}})^{\top} = (X^{\top})_{Y^{\top}}^{\otimes_{k,\ell}}$,
		\item $(X_{Y}^{\otimes_{k,\ell}})^{-1} = (X^{-1})_{Y^{-1}}^{\otimes_{k,\ell}}$,
	\end{enumerate}
	
	\vspace{1mm}
	
	\noindent where with a slight abuse of notation,
	\begin{align*}
		(XV)_{(YW)}^{\otimes_{k,\ell}} = X_1V_1 &\otimes\dots\otimes X_{k-1}V_{k-1} \otimes \\ & YW \otimes X_{k+1}V_{k+1} \otimes\dots\otimes X_{\ell}V_{\ell}.
	\end{align*}
	
	\begin{lemma}
		\label{lem:dynamics}
		The dynamics in \eqref{eq:overallDynamics3} can be written as
		\begin{align}
			\dot{\vec{x}} = \big( A_{\oplus} + \Delta A_{\oplus} \big) \vec{x} + B_{\otimes} \vec{u},
			\label{eq:dynamicsNew}
		\end{align}
		where $A_{\oplus} = \overset{m}{\underset{i=1}{\oplus}} A_i$ and $\Delta A_{\oplus} = \overset{m}{\underset{i=1}{\oplus}} \Delta A_i$.
	\end{lemma}
	
	\vspace{3mm}

	\noindent The proof of \Cref{lem:dynamics} is straightforward using induction and the properties of Kronecker products.
	We assume that the generalized perturbation $\Delta A_{\oplus}$ represents a structured uncertainty composed of the perturbations from each layer of the system.
	Generalization of a layer-independent perturbation or leveraging other well-known uncertainty structures are addressed for future works.
	
	
	\vspace{2mm}

	\vspace{2mm}
	
	\begin{theorem}
		\label{thm:generalizedWithControl}
		Consider the generalized dynamics in \eqref{eq:dynamicsNew}.
		Assume that $Q_1\succeq 0$, $R_1\succ0$, and symmetric positive function $\mathcal{U}_1$ is given such that
		\begin{align}
			\Delta A_1^{\top}\bar{P} + \bar{P}\Delta A_1 \preceq \mathcal{U}_1(\bar{P}),
			\label{eq:inequalityGeneral}
		\end{align}
		holds for all $\bar{P}\succ0$ and $\Delta A_1$.
		Furthermore, let $P_1\succ0$ be the solution to,
		\begin{align*}
			A_1^{\top}P_1 + P_1A_1 + Q_1 - P_1B_1R_1^{-1}B_1^{\top}P_1 + \mathcal{U}_1(P_1) = 0,
		\end{align*}
		and define,
		\begin{align*}
			F_i = A_i^{\top} M_i + M_i A_i, \qquad G_i = \Delta A_i^{\top} M_i + M_i \Delta A_i,
		\end{align*}
		where $M_i\succ 0$ is such that $F_i\preceq 0$ and $G_i\prec 0$ for all $i=2,\dots,m$.
		Then the generalized state-feedback control law $\vec{u}=-K_{\otimes}\vec{x}$ with $K_{\otimes} = K_1\otimes I_{n_2}\otimes\dots\otimes I_{n_m}$ implies that,
		\begin{align}
			\bar{J} = \int_{0}^{\infty} \big( \vec{x}^{\top}Q_{\otimes}\vec{x} + \vec{u}^{\top}R_{\otimes}\vec{u} \big) dt \ \leq \ \vec{x}_0^{\top}P_{\otimes}\vec{x}_0,
		\end{align}
		where $P_{\otimes}$, $Q_{\otimes}$, and $R_{\otimes}$ are defined as,
		\begin{align*}
			P_{\otimes} &= M^{\otimes_{1,m}}_{P_1}, \qquad R_{\otimes} = M_{R_1}^{\otimes_{1,m}}, \\
			Q_{\otimes} &= M^{\otimes_{1,m}}_{Q_1} - P_1\otimes \big( \sum_{i=2}^m M_{F_i}^{\otimes_{i,m}} \big).
		\end{align*}
	\end{theorem}
	
	
	\begin{proof}
		We proceed by checking the conditions of \Cref{thm:initialControl} but for the layered system in \eqref{thm:generalizedWithControl}.
		To this end, we need a new definition for a symmetric positive function that generalizes $\mathcal{U}$.
		Let,
		\begin{align}
			\mathcal{V}(T_1,T_2,\dots,T_{\ell}) = T_{\mathcal{U}_1(T_1)}^{\otimes_{1,\ell}}.
			\label{eq:newFunction}
		\end{align}
		Then we note that for $m=2$,
		\begin{align*}
			\Delta A_{\oplus}^{\top}P_{\otimes}
			&= (\Delta A_1\oplus \Delta A_2)^{\top}(P_1\otimes M_2) \\
			&= (\Delta A_1^{\top}\otimes I_{n_2} + I_{n_1}\otimes \Delta A_2^{\top})(P_1\otimes M_2) \\
			&= \Delta A_1^{\top}P_1\otimes M_2 + P_1\otimes \Delta A_2^{\top}M_2 ~.
		\end{align*}
		Similarly,
		\begin{align*}
			P_{\otimes}\Delta A_{\oplus} = P_1\Delta A_1\otimes M_2 + P_1\otimes M_2\Delta A_2.
		\end{align*}
		Hence by induction, it can be shown that for any $m$,
		\begin{equation}
			\begin{aligned}
				\Delta A_{\oplus}^{\top}&P_{\otimes} + P_{\otimes}\Delta A_{\oplus} \\
				&= \big( \Delta A_1^{\top}P_1 + P_1\Delta A_1 \big)\otimes M_2\otimes\dots\otimes M_m \\
				& \hspace{40mm} + P_1\otimes \big( \sum_{i=2}^m M_{F_i}^{\otimes_{i,m}} \big) \\
				&\preceq\ \mathcal{U}_1(P_1)\otimes M_2\otimes\dots\otimes M_m \\
				&=\ \mathcal{V}(P_1,M_2,\dots,M_m),
				\label{eq:generalUpperBound}
			\end{aligned}
		\end{equation}
		where we have used \eqref{eq:inequalityGeneral} and the fact that Kronecker products preserve positive-definiteness.
		Also from Proposition 1.3, it is straightforward to show that $\mathcal{V}$ is a symmetric and positive.
		From Proposition 1.4, $R_{\otimes}^{-1}=(M^{-1})_{R_1^{-1}}^{\otimes_{1,m}}$; hence,
		\begin{align}
			\label{eq:finalGeneralTerm}
			P_{\otimes}B_{\otimes}R_{\otimes}^{-1}B_{\otimes}^{\top}P_{\otimes} = M_{P_1B_1R_1^{-1}B_1^{\top}P_1}^{\otimes_{1,m}},
		\end{align}
		which gives,
		\begin{align*}
		A_{\oplus}^{\top}P_{\otimes} &+ P_{\otimes}A_{\oplus} + Q_{\otimes} \\
		&-P_{\otimes}B_{\otimes}R_{\otimes}^{-1}B_{\otimes}^{\top}P_{\otimes} + \mathcal{V}(P_1,M_2,\dots,M_m) \\
		&= M_{A_1^{\top}P_1 + P_1A_1 + Q_1 + P_1B_1R_1^{-1}B_1^{\top}P_1 + \mathcal{U}_1(P_1)}^{\otimes_{1,m}} = 0,
		\end{align*}
		and from proposition 1.2,
		\begin{align*}
			K_{\otimes} = -R^{-1}_{\otimes} B_{\otimes}^{\top} P_{\otimes} = K_1\otimes I_{n_2}\otimes\dots\otimes I_{n_m}.
		\end{align*}
	\end{proof}
	
	\vspace{2mm}
	
	There are some remarks needed in relation to~\Cref{thm:generalizedWithControl}.
	First, from the definitions of $F_i$ and $G_i$, the perturbed dynamics $A_i+\Delta A_i$ is implicitly assumed to be stable which is not necessarily required.
	In this sense, the assumptions $M_i\succ 0$, $F_i\preceq 0$, and $G_i\prec 0$ might be restrictive.
	Indeed, we need $M_i$'s to be selected in a way that inequalities such as \eqref{eq:generalUpperBound} and $Q_{\otimes} \succeq 0$ hold which may require further assumptions on the structure of $M_i$ such as being diagonal or sparse.
	This also limits the freedom of the designer to only select the matrices $M_i$ while forming the cost of the LQR problem.
	
	Moreover, as mentioned in \Cref{sec:guaranteedLQ}, we assume a layered structure for the controller where the input to the first layer is repeated in the subsequent layers reflected into the Kronecker structure.
	While this assumption reduces system couplings, the presence of the other layers' dynamics is implicit in parameters $F_i$ and $G_i$.
	
	Finally, the proposed $Q_{\otimes}$ essentially removes the couplings of the dynamics of different layers that shows up in $A_{\oplus}^{\top}P_{\otimes} + P_{\otimes}A_{\oplus}$ in the problem formulation.
	However, it needs to be verified whether this $Q_{\otimes}$ satisfies the existence and stabilizability criteria of the LQR solution.
	To that end, it is straightforward to check that $Q_{\otimes}\succeq0$; in fact $Q_{\otimes}=L^{\top}L$ where,
	\begin{align*}
		L=\begin{bmatrix}
			D\otimes M_2^{1/2}\otimes M_3^{1/2} \otimes\dots\otimes M_m^{1/2} \\
			H\otimes N_2\otimes M_3^{1/2} \otimes\dots\otimes M_m^{1/2} \\
			H\otimes M_2^{1/2}\otimes N_3 \otimes\dots\otimes M_m^{1/2} \\
			\vdots \\
			H\otimes M_2^{1/2}\otimes M_3^{1/2} \otimes\dots\otimes N_m
		\end{bmatrix},
	\end{align*}
	and $Q_1=D^{\top}D$, $P_1=H^{\top}H$, and $A_i^{\top}M_i + M_iA_i=N_i^{\top}N_i$ by Cholesky decomposition.
	Hence to obtain the stability of the generalized LQR solution, we need the implicit assumption that $(A_{\oplus},L)$ is observable (via proper choices of $M_i$'s) and the controllability of $(A_{\oplus},B_{\otimes})$ (discussed in \cite{chapman2014controllability}).

	\vspace{1mm}
	
	\section{Compositional Synthesis for Social Networks}
	\label{sec:sims}
	
	\vspace{1mm}
	
	Layered networks can be used for modeling geographical distribution of various social types.
	In this section we implement the guaranteed-cost compositional design on a social influence network.
	This case study is inspired by Padgett's research on 15 elite families in 1282-1500, Florence \cite{padgett2011introduction} and the impact of Renaissance on Italian art and culture in the same time interval.
	The analysis provides a grouping of these families into social, political, business, and financial members
	and the interactions between families were limited to these corresponding members.
	Based on the geographical distribution and ties between these families, we leverage our methodology to model this multi-layered network.
	Inherently, modeling such an organization is challenging due to the complexity of societal interactions as well as the population size.
	We account for these types of uncertainties in parameterizing the network dynamics.
	These uncertainties can potentially lead to misclassification in the nature of connections among individuals.
	\begin{figure}[h]
		\centering
		\begin{minipage}{0.4\linewidth}
			\centering
			\subfloat[]{\label{fig:elite1}\includegraphics[width=\linewidth]{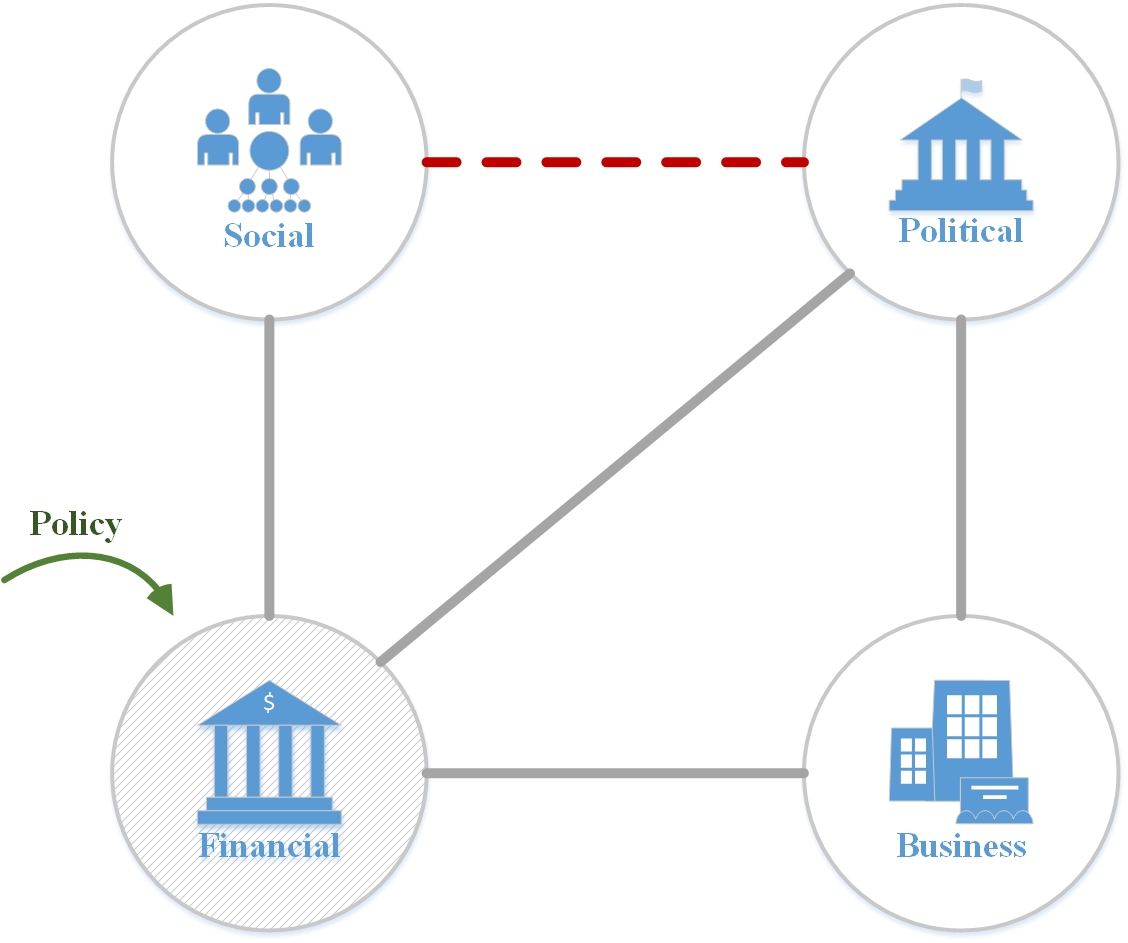}}
		\end{minipage}
		\hspace{3mm}
		\begin{minipage}{0.5\linewidth}
			\centering
			\subfloat[]{\label{fig:elite2}\includegraphics[width=\linewidth]{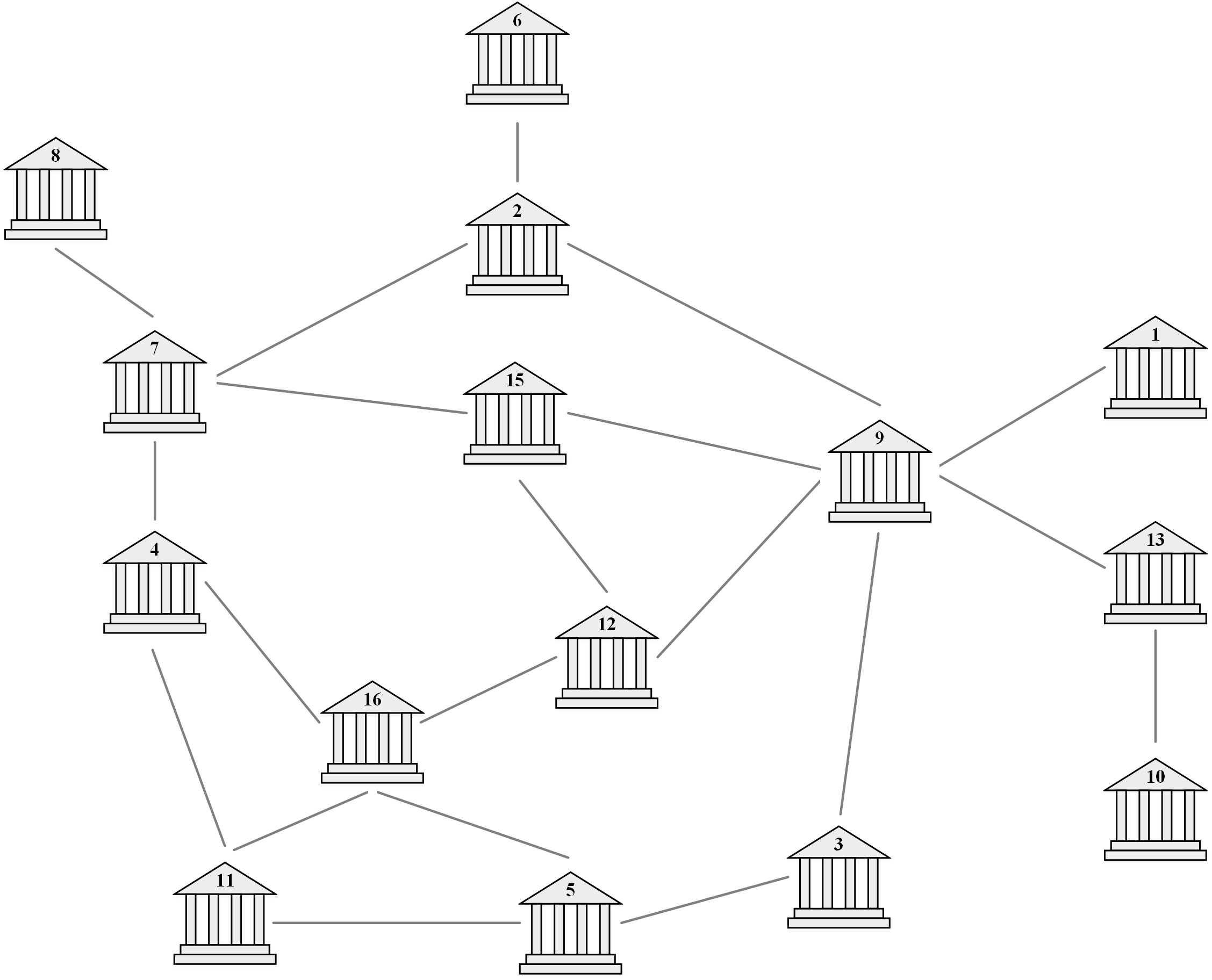}}
		\end{minipage}
		\caption{The elite family layered structures (a) interconnections within each family (b) connections among all families}
		\label{fig:plotElite}
	\end{figure}
		
	Our model contains three different layers: The first layer, $\mathcal{G}_1$, represents the structure of each family containing the four groups (Fig. \ref{fig:elite1}).
	The dashed line denotes a negative edge denoting a disagreement between social and political entities.
	An input to the financial member of each family is considered in order to both react to a change in fiscal strategy in response to Renaissance fluctuations and avoid social cleavage due to the opposition between two main members of the family.\footnote{In LQR terminology, we only have access to the financial control knob to bound the system performance. This is just a simplified assumption and the control can take place on every node.}
	All connections are assumed to be equal (not weighted).
	The perturbation to the system comes from a mistakenly flipped sign of the connection between social and financial groups.
	This results in clustering leading to the instability of the system.\footnote{For more on clustering and controllability in signed networks the reader is referred to \cite{altafini2013consensus,alemzadeh2017controllability}.}
	The Florentine elite families graph, $\mathcal{G}_2$, designates the second layer of the network (Fig. \ref{fig:elite2}).
	The third layer is inspired by the spread of Renaissance throughout other provinces of Italy such as Rome and Venice (Fig. \ref{fig:italy}).
	This extra layer signifies the computational efficiency of the method.
	\begin{figure}[h]
		\centering
		\includegraphics[scale=0.11]{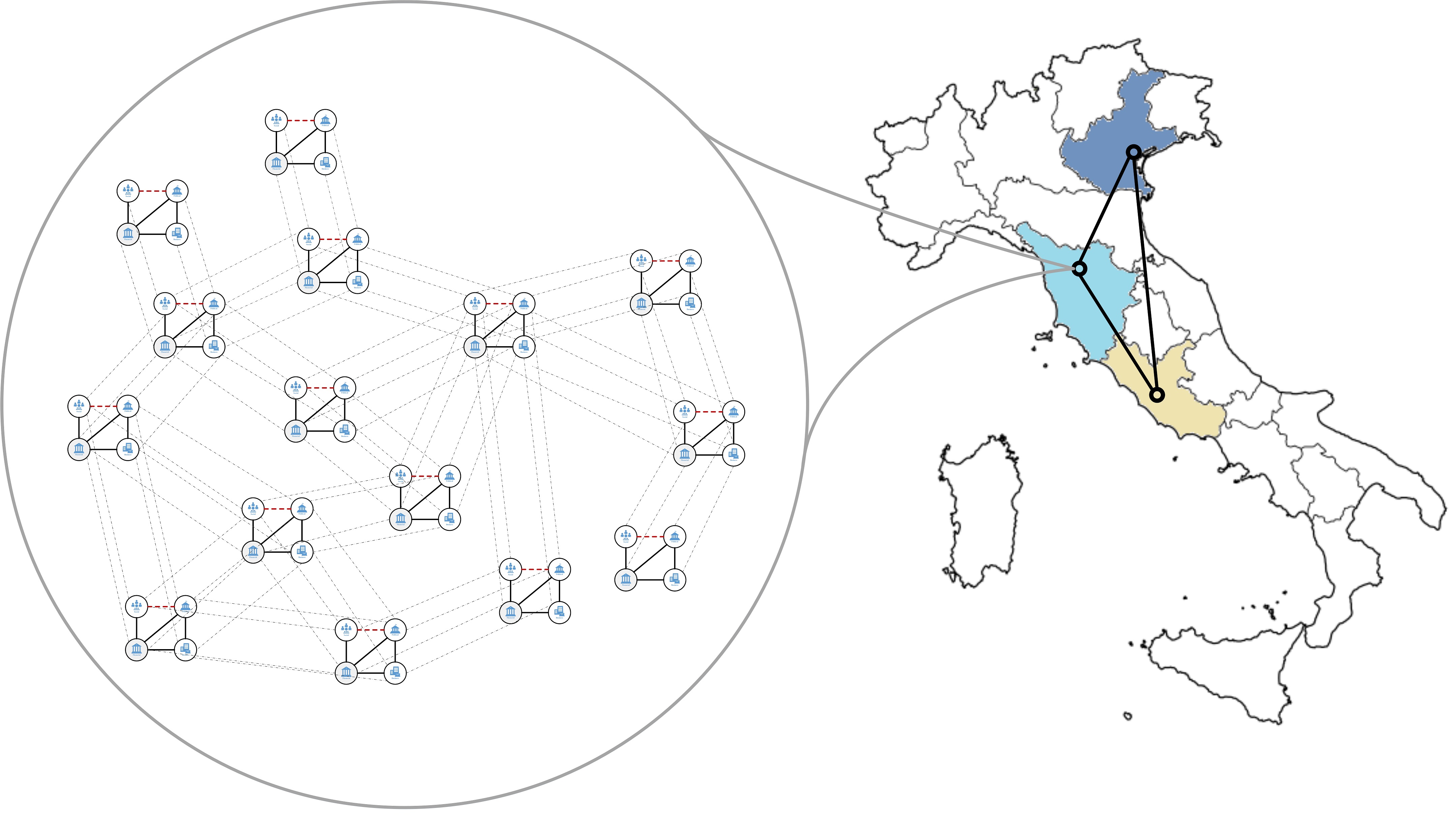}
		\caption{Composition of the elite families network layers.}
		\label{fig:italy}
	\end{figure}
	
	We use Taylor's model of opinion evolution as discussed in \Cref{sec:prob-setup}.
	In particular, we use \Cref{eq:overallDynamics3} to model this 3-layer dynamics as,
	\begin{align*}
		\dot{\vec{x}} = -\big[ \overset{3}{\underset{i=1}{\oplus}} (A_i + \Delta A_i ) \big] \vec{x} + ( \overset{3}{\underset{i=1}{\otimes}} B_i )\vec{u},
	\end{align*}
	where $A_i$ and $\Delta A_i$ denote the Laplacian and the uncertainty matrices of layer $i$.
	We assume $\Delta \mathcal{L}_2 = 0$, $\Delta \mathcal{L}_3 = 0$, and $B_1 = [0\quad 0\quad 0\quad 1]^{\top}$ reflecting the control over the financial node.
	We use \eqref{eq:perturbation} to model the perturbation with $d=1$, $w_1=2$, and $\tilde{A}_1 = e_1e_2^{\top} + e_2e_1^{\top}$, i.e., a change in the sign between social and political groups.
	\begin{figure}[h]
		\begin{minipage}{0.49\linewidth}
			\subfloat[]{\label{fig:cost1}\includegraphics[width=\linewidth]{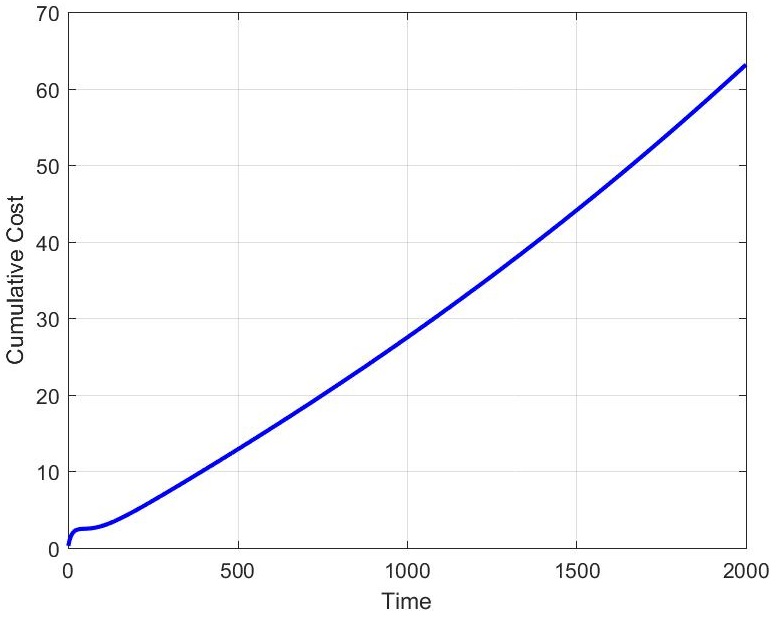}}
		\end{minipage}
		\hspace{-1mm}
		\begin{minipage}{0.49\linewidth}
			\subfloat[]{\label{fig:cost2}\includegraphics[width=\linewidth]{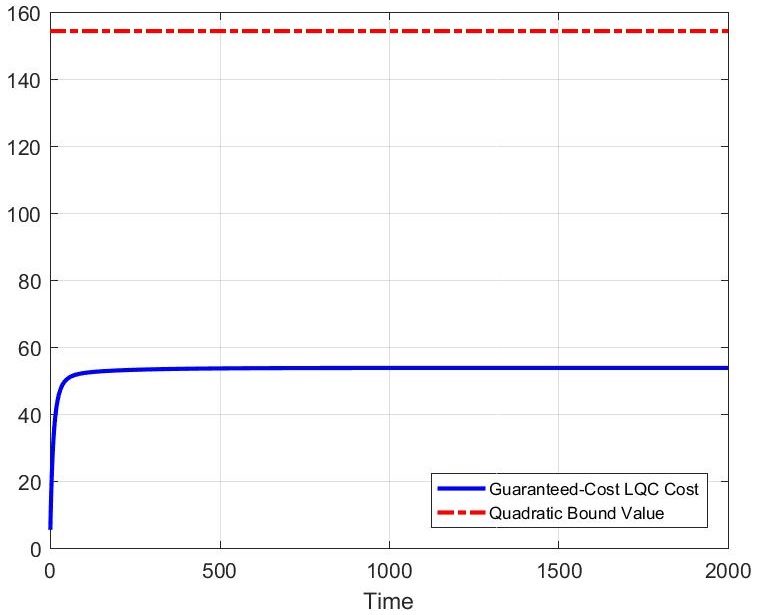}}
		\end{minipage}
		\caption{LQ system performance (a) Normal LQR algorithm (b) Guaranteed-cost LQ controller design}
		\label{fig:costPlot}
	\end{figure}
		
	Fig. \ref{fig:cost1} depicts the instability of the system when the baseline LQR algorithm is used without taking the uncertainties into account.
	Fig. \ref{fig:cost2} shows the guaranteed performance for a similar setup but with an updated LQ controller design methodology presented in \Cref{sec:analysis}.
	\Cref{lab:time} shows the time it takes to run the LQR algorithm (updated ARE in particular) based on the size of $\mathcal{G}_3$ (number of provinces).
	Similar results can also be obtained for $\mathcal{G}_1$ and $\mathcal{G}_2$.
	\begin{table}[H]
		\centering
		\resizebox{0.35\textwidth}{!}{
		\begin{tabular}{|c||c|c|c|c|}
			\hline
			Size of $\mathcal{G}_3$ & 1 & 2 & 3 & 4 \\ [0.5ex] 
			\hline
			Time (sec) & 0.3 & 2 & 216 & 419 \\ [0.5ex]
			\hline
		\end{tabular}
		}
		\caption{Computational performance for solving modified ARE for different sizes of $\mathcal{G}_3$}
		\label{lab:time}
	\end{table}
	
	\section{CONCLUSION}
	\label{sec:conclusion}
	
	In this paper we examine large-scale social networks analysis via a system-theoretic compositional approach.
	In particular, we consider the guaranteed-cost control design of specific classes of social systems with a layered structure.
	This specific structure enables the guaranteed performance control and estimation based on the factors.
	We present a compositional approach to determine these guarantees considering the uncertainties that originate from the network layers.
	The applicability of the developed theory to large-scale social networks with misclassified interdependencies is then discussed.
	Some future directions include generalization of the dynamics such that every layer can potentially contribute to the control mechanism for the system.
	One can also aim to provide a more topological structure to the design parameters $Q_{\otimes}$ and $R_{\otimes}$.

	\section*{ACKNOWLEDGMENTS}
	
	S. Alemzadeh would like to thank Airlie Chapman and Behnoosh Parsa for insightful conversations and suggestions pertaining to this work.
	The authors also acknowledge many fruitful discussions with Hesam Talebiyan, Andr{\'e}s Gonz{\'a}lez, Leonardo Due{\~n}as-Osorio, and Rassia D'Souza on the applications of control theory to social and infrastructure networks.
	
	
	\bibliographystyle{ieeetr}
	\bibliography{citations2}
	
	
\end{document}